\documentclass[a4paper,11pt]{article}
\usepackage{amsthm}
\usepackage{amssymb}

\usepackage{amsmath}
\usepackage{amssymb}
\usepackage{graphics}
\usepackage{tikz}
\usetikzlibrary{arrows,decorations.markings,decorations.pathmorphing,backgrounds,positioning,fit,matrix}
\usepackage{pgf}
\usepackage{pgfplots}

\usepackage{euscript}
\usepackage{comment}
\usepackage{textgreek}

 \newcommand{\R}{{\mathbb R}}
 \newcommand{\C}{{\mathbb C}}

\newcommand{\tildebaja}{{\raise.17ex\hbox{$\scriptstyle\sim$}}}

\newcommand{\dist}{{\rm dist}}
\newcommand{\proy}{\pi}

\newcommand{\Dd}{{\EuScript D}}

\newcommand{\Tt}{{\EuScript T}}
\newcommand{\Qq}{{\EuScript Q}}

\newcommand{\Rr}{{\EuScript R}}
\newcommand{\Ss}{{\EuScript S}}

\newcommand{\Int}{\operatorname{Int}}

\theoremstyle{plain}
\newtheorem{thm}{Theorem}[section]
\newtheorem{prop}[thm]{Proposition}

\newtheorem{lem}[thm]{Lemma}
\newtheorem{prob}[thm]{Problem}

\theoremstyle{definition}
\newtheorem{defi}[thm]{Definition}

\theoremstyle{remark}

\newtheorem{remark}[thm]{Remark}

{\em}

{\em}

\newcounter{substep}
\def\thesubstep{\arabic{substep}}

{\em}

\newcounter{subsubstep}
\def\thesubsubstep{\arabic{subsubstep}}

{\em}

\begin{document}
\title{The open quadrant problem: A topological proof}
\author{Jos\'e F. Fernando, J. M. Gamboa, Carlos Ueno\thanks{First author supported by Spanish GAAR MTM2011-22435, Grupos UCM 910444 and the ``National Group for Algebraic and Geometric Structures, and their Applications'' (GNASA - INdAM). His one year research stay in the Dipartimento di Matematica of the Universit\`a di Pisa is partially supported by MECD grant PRX14/00016. Second author supported by Spanish GAAR MTM2011-22435 and Grupos UCM 910444. Third author supported by `Scuola Galileo Galilei' Research Grant at the Dipartimento di Matematica of the Universit\`a di Pisa and Spanish GAAR MTM2011-22435.}\\[0.5cm]
Departamento de \'Algebra, Facultad de Ciencias Matem\'aticas\\
Universidad Complutense de Madrid \\ 28040 MADRID (SPAIN)\\
{\tt josefer@mat.ucm.es}\ \ \ \ \ \ {\tt jmgamboa@mat.ucm.es} \\[0.5cm]
Dipartimento di Matematica, Universit\`a di Pisa\\ 56127 PISA (ITALY)\\
{\tt jcueno@mail.dm.unipi.it}}
\date{}

\maketitle

\begin{abstract}
\noindent In this work we present a new polynomial map $f:=(f_1,f_2):\R^2\to\R^2$ whose image is the open quadrant $\Qq:=\{x>0,y>0\}\subset\R^2$. The proof of this fact involves arguments of topological nature that avoid hard computer calculations. In addition each polynomial $f_i\in\R[{\tt x},{\tt y}]$ has degree $\leq16$ and only $11$ monomials, becoming the simplest known map solving the open quadrant problem.\\[4pt] 
\noindent\em Keywords\em. Polinomal map, polynomial image, semialgebraic set, open quadrant.\\[2pt]
2000 \em Mathematics Subject Classification\em. 14P10, 26C99(primary); 52A10 (secondary).
\end{abstract}

\section{Introduction}

Although it is usually said that the first work in Real Geometry is due to Harnack \cite{h}, who obtained an upper bound for the number of connected components of a non-singular real algebraic curve in terms of its genus, modern Real Algebraic Geometry was born with Tarski's article \cite{ta}, where it is proved that the image of a semialgebraic set under a polynomial map is a semialgebraic set. We are interested in studying what might be called the `inverse problem'. In the 1990 \em Oberwolfach Reelle algebraische Geometrie \em week \cite{g} the second author proposed: 

\begin{prob}\label{prob0}
Characterize the (semialgebraic) subsets of $\R^m$ that are either polynomial or regular images of $\R^n$. 
\end{prob}

A map $f:=(f_1,\ldots,f_m):\R^n\to\R^m$ is a \em polynomial map \em if its components $f_k\in\R[{\tt x}]:=\R[{\tt x}_1,\ldots,{\tt x}_n]$ are polynomials. Analogously, $f$ is a \em regular map \em if its components can be represented as quotients $f_k=\frac{g_k}{h_k}$ of two polynomials $g_k,h_k\in\R[{\tt x}]$ such that $h_k$ never vanishes on $\R^n$. A subset $S\subset\R^n$ is \em semialgebraic \em when it admits a description by a finite boolean combination of polynomial equalities and inequalities.

Open semialgebraic sets deserve a special attention in connection with the real Jacobian Conjecture \cite{p}. In particular the second author stated in \cite{g} the `open quadrant problem':

\begin{prob}\label{prob}
Determine whether the open quadrant $\Qq:=\{x>0,y>0\}$ of $\R^2$ is a polynomial image of $\R^2$.
\end{prob}

This problem stimulated the interest of many specialists in the field. However, only after twelve years a first solution was found in \cite{fg1} and presented by the first author in the 2002 \em Oberwolfach Reelle algebraische Geometrie \em week \cite{fe1}.

The open quadrant problem was the germ of a more systematic study of `Polynomial and regular images of Euclidean spaces' developed by the authors during the last decade and which was the topic of the Ph.D. Thesis of the third author \cite{u0}. Since then we have worked on this issue with two main objectives: 
\begin{itemize}
\item Finding obstructions to be an either polynomial or regular image. 
\item Proving (constructively) that large families of semialgebraic sets with piecewise linear boundary (convex polyhedra, their interiors, complements and the interiors of their complements) are either polynomial or regular images of some Euclidean space. The positive answer to the open quadrant problem has been a recurrent starting point for this approach.
\end{itemize}

In \cite{fg1,fg2} we presented the first steps to approach Problem \ref{prob0}. A complete solution to Problem \ref{prob0} for the one-dimensional case appears in \cite{fe2}, whereas in \cite{fgu1,fu1,fu2,u1,u2} we approached constructive results concerning the representation as either polynomial or regular images of the semialgebraic sets with piecewise linear boundary commented above. Articles \cite{fgu2,fu3} are of different nature because we find in them new obstructions for a subset of $\R^m$ to be either a polynomial or a regular image of $\R^n$. In the first one we found some properties of the difference ${\rm Cl}(S)\setminus S$ while in the second it is shown that \em the set of points at infinite of a polynomial image of $\R^n$ is a connected set\em.

The constructive solution to the open quadrant problem provided in \cite{fg1} involves quite complicated computer calculations that the third author never liked. In fact he provided in his Ph.D. Thesis a different topological proof for the map proposed in \cite{fg1}, together with an algebraic proof involving a different polynomial map. This map has inspired the first and third authors for a short algebraic proof of the open quadrant problem involving a new polynomial map \cite{fu4} and has led us to look for a polynomial map with optimal algebraic structure whose image is the open quadrant. It is important to establish clearly the meaning of `optimal algebraic structure' \cite[\S3(A)]{fu4}. It is natural to wonder how a polynomial map looks like when completely expanded and how it compares with other polynomial maps. We care about the total degree of the involved polynomial map (the sum of the degrees of its components) and its total number of (non-zero) monomials. We would like to find a polynomial map with the less possible total degree and the less possible number of monomials. The example in \cite{fg1} has total degree $56$ and its total number of monomials is $168$. The polynomial map in \cite{fu4} has total degree $72$ and its total number of monomials is $350$. In this work we will prove:

\begin{thm}\label{main}
The open quadrant $\Qq$ is the image of the polynomial map
$$
f:\R^2\to\R^2,\ (x,y)\mapsto((x^2y^4+x^4y^2-y^2-1)^2+x^6y^4,(x^6y^2+x^2y^2-x^2-1)^2+x^6y^4).
$$
\end{thm}

This polynomial map has total degree $28$ and its total number of monomials is $22$, which certainly improves the already known explicit solutions to the open quadrant problem. It has been constructed following a similar strategy to that in \cite[\S3]{fg1}. Our experience approaching this problem suggests us that this map is surely close to have the optimal desired algebraic structure.

The article is organized as follows. In Section \ref{s2} we present all basic notions and topological preliminaries used in Section \ref{s3} to prove Theorem \ref{main}.

\section{Topological preliminaries}\label{s2}

Denote the closed disc of center the origin and radius $A>0$ of the plane $\R^2$ with ${\mathbb D}_A$. A \em warped disc \em is a subset $\Dd_{A,\xi}:=\{z=\xi(x,y),\ x^2+y^2\le A^2\}\subset\R^3$ where $\xi:\R^2\to\R$ is a continuous function. Consider the homeomorphism
$$
\zeta:\R^3\to\R^3,\ (x,y,z)\mapsto(x,y,z-\xi(x,y))
$$
that maps $\Dd_{A,\xi}$ onto ${\mathbb D}_A\times\{0\}$. The image of $\Dd_{A,\xi}$ under a permutation of the variables of $\R^3$ will be also called a warped disc.


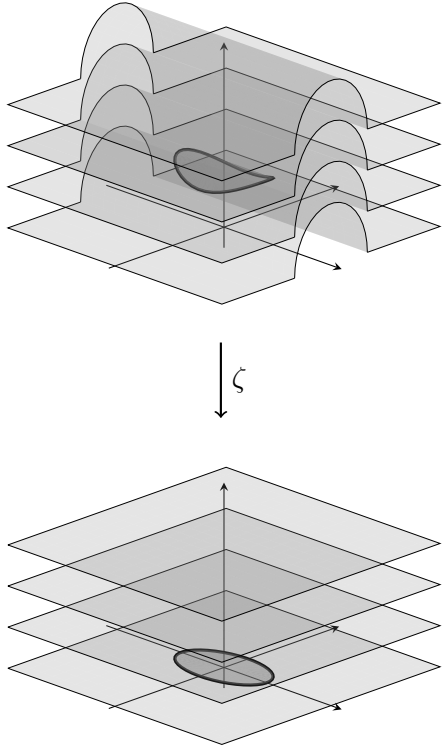
\begin{figure}[!ht]
\begin{center}
\begin{tabular}{c}
\begin{tikzpicture}[trim right=6.5cm]
\pgfplotsset{ticks=none}
\begin{axis}[view={45}{30}, scale=1.8, axis lines=middle, footnotesize,xlabel=$$,ylabel=$$,
xmin=-5.5,xmax=5.5,ymin=-9.5,ymax=9.5,zmin=-1,
zmax=9, unit vector ratio=]
\addplot3[opacity=0.0, fill opacity=0.2, surf, gray, samples=5,
variable=\u, variable y=\v,
domain=-5:5, y domain=-9:-3]
({u}, {v}, {0});
\addplot3[opacity=0.0, fill opacity=0.2, surf, gray, samples=5,
variable=\u, variable y=\v,
domain=-5:5, y domain=3:9]
({u}, {v}, {0});
\addplot3[opacity=0.0, fill opacity=0.2, surf, gray, samples=20,
variable=\u, variable y=\v,
domain=-5:5, y domain=0:180]
({u}, {3*cos(v)}, {3*sin(v)});
\addplot3[very thin,mark=,black] plot coordinates {
          (-5,-3,0)
          (-5,-9,0)
          (5,-9,0)
          (5,-3,0)};
\addplot3[very thin,mark=,black] plot coordinates {
          (-5,3,0)
          (-5,9,0)
          (5,9,0)
          (5,3,0)};
\addplot3[very thin,domain=0:pi,samples=20,samples y=0]
({-5},
{3*cos(deg(x))},
{3*sin(deg(x))});
\addplot3[very thin,domain=0:pi,samples=20,samples y=0]
({5},
{3*cos(deg(x))},
{3*sin(deg(x))});
\addplot3[very thick,domain=0:2*pi,samples=40,samples y=0]
({2*cos(deg(x))},
{2*sin(deg(x))},
{sqrt(9-4*(sin(deg(x)))^2)});
\addplot3[opacity=0, fill opacity=0.5, surf,samples=40,gray,
variable=\u, variable y=\v,
domain=0:2, y domain=0:360]
({u*cos(v)},{u*sin(v)},{sqrt(9-u^2*(sin(v))^2)});
\addplot3[opacity=0.0, fill opacity=0.2, surf,gray,samples=5,
variable=\u, variable y=\v,
domain=-5:5, y domain=-9:-3]
({u}, {v}, {2});
\addplot3[opacity=0.0, fill opacity=0.2, surf, gray,samples=5,
variable=\u, variable y=\v,
domain=-5:5, y domain=3:9]
({u}, {v}, {2});
\addplot3[opacity=0.0, fill opacity=0.2, surf, gray, samples=20,
variable=\u, variable y=\v,
domain=-5:5, y domain=0:180]
({u}, {3*cos(v)}, {2+3*sin(v)});
\addplot3[very thin,domain=0:pi,samples=20,samples y=0]
({-5},
{3*cos(deg(x))},
{2+3*sin(deg(x))});
\addplot3[very thin,domain=0:pi,samples=20,samples y=0]
({5},
{3*cos(deg(x))},
{2+3*sin(deg(x))});
\addplot3[very thin,mark=,black] plot coordinates {
          (-5,-3,2)
          (-5,-9,2)
          (5,-9,2)
          (5,-3,2)};
\addplot3[very thin,mark=,black] plot coordinates {
          (-5,3,2)
          (-5,9,2)
          (5,9,2)
          (5,3,2)};
\addplot3[opacity=0.0, fill opacity=0.3, surf,gray,samples=5,
variable=\u, variable y=\v,
domain=-5:5, y domain=-9:-3]
({u}, {v}, {4});
\addplot3[opacity=0.0, fill opacity=0.2, surf, gray,samples=5,
variable=\u, variable y=\v,
domain=-5:5, y domain=3:9]
({u}, {v}, {4});
\addplot3[opacity=0.0, fill opacity=0.2, surf, gray, samples=20,
variable=\u, variable y=\v,
domain=-5:5, y domain=0:180]
({u}, {3*cos(v)}, {4+3*sin(v)});
\addplot3[very thin,domain=0:pi,samples=20,samples y=0]
({-5},
{3*cos(deg(x))},
{4+3*sin(deg(x))});
\addplot3[very thin,domain=0:pi,samples=20,samples y=0]
({5},
{3*cos(deg(x))},
{4+3*sin(deg(x))});
\addplot3[very thin,mark=,black] plot coordinates {
          (-5,-3,4)
          (-5,-9,4)
          (5,-9,4)
          (5,-3,4)};
\addplot3[very thin,mark=,black] plot coordinates {
          (-5,3,4)
          (-5,9,4)
          (5,9,4)
          (5,3,4)};
\addplot3[opacity=0.0, fill opacity=0.2, surf,gray,samples=5,
variable=\u, variable y=\v,
domain=-5:5, y domain=-9:-3]
({u}, {v}, {6});
\addplot3[opacity=0.0, fill opacity=0.2, surf, gray,samples=5,
variable=\u, variable y=\v,
domain=-5:5, y domain=3:9]
({u}, {v}, {6});
\addplot3[opacity=0.0, fill opacity=0.2, surf, gray, samples=20,
variable=\u, variable y=\v,
domain=-5:5, y domain=0:180]
({u}, {3*cos(v)}, {6+3*sin(v)});
\addplot3[very thin,domain=0:pi,samples=20,samples y=0]
({-5},
{3*cos(deg(x))},
{6+3*sin(deg(x))});
\addplot3[very thin,domain=0:pi,samples=20,samples y=0]
({5},
{3*cos(deg(x))},
{6+3*sin(deg(x))});
\addplot3[very thin,mark=,black] plot coordinates {
          (-5,-3,6)
          (-5,-9,6)
          (5,-9,6)
          (5,-3,6)};
\addplot3[very thin,mark=,black] plot coordinates {
          (-5,3,6)
          (-5,9,6)
          (5,9,6)
          (5,3,6)};
\end{axis}
\end{tikzpicture}\\
\begin{tikzpicture}
\draw [thick, ->] (0,0)--(0,-0.5) node [right] {$\zeta$} -- (0,-1);
\end{tikzpicture}\\[-0.4cm]
\begin{tikzpicture}[trim right=6.5cm]
\pgfplotsset{ticks=none}
\begin{axis}[view={45}{30}, scale=1.8, axis lines=middle, footnotesize,xlabel=$$,ylabel=$$,
xmin=-5.5,xmax=5.5,ymin=-9.5,ymax=9.5,zmin=-1,
zmax=9, unit vector ratio=]
\addplot3[opacity=0.0, fill opacity=0.2, surf,gray,samples=20,colormap/hot,
variable=\u, variable y=\v,
domain=-5:5, y domain=-9:9]
({u}, {v}, {0});
\addplot3[very thin,mark=,black] plot coordinates {
          (-5,-9,0)
          (-5,9,0)
          (5,9,0)
          (5,-9,0)
          (-5,-9,0)};
\addplot3[very thick,domain=0:2*pi,samples=40,samples y=0]
({2*cos(deg(x))},
{2*sin(deg(x))},
{0});
\addplot3[opacity=0, fill opacity=0.5, surf,samples=40,gray,
variable=\u, variable y=\v,
domain=0:2, y domain=0:360]
({u*cos(v)},{u*sin(v)},{0});
\addplot3[opacity=0.0, fill opacity=0.2, surf,gray,samples=20,colormap/hot,
variable=\u, variable y=\v,
domain=-5:5, y domain=-9:9]
({u}, {v}, {2});
\addplot3[very thin,mark=,black] plot coordinates {
          (-5,-9,2)
          (-5,9,2)
          (5,9,2)
          (5,-9,2)
          (-5,-9,2)};
\addplot3[opacity=0.0, fill opacity=0.2, surf,gray,samples=20,colormap/hot,
variable=\u, variable y=\v,
domain=-5:5, y domain=-9:9]
({u}, {v}, {4});
\addplot3[very thin,mark=,black] plot coordinates {
          (-5,-9,4)
          (-5,9,4)
          (5,9,4)
          (5,-9,4)
          (-5,-9,4)};
\addplot3[opacity=0.0, fill opacity=0.2, surf,gray,samples=20,colormap/hot,
variable=\u, variable y=\v,
domain=-5:5, y domain=-9:9]
({u}, {v}, {6});
\addplot3[very thin,mark=,black] plot coordinates {
          (-5,-9,6)
          (-5,9,6)
          (5,9,6)
          (5,-9,6)
          (-5,-9,6)};
\end{axis}
\end{tikzpicture}
\end{tabular}
\caption{The homeomorphism $\zeta$ for $\xi(x,y):=\sqrt{B^2-\min(y^2,B^2)}$ acting on $\R^3$.}\label{pic:homeo}
\end{center}
\end{figure}


For each $\varepsilon>0$ consider the open neighborhood 
$$
{\mathbb D}_A(\varepsilon):=\{x^2+y^2<(A+\varepsilon)^2\}\times(-\varepsilon,\varepsilon)\subset\R^3
$$ 
of ${\mathbb D}_A$. Clearly, $\Dd_{A,\xi}(\varepsilon):=\zeta^{-1}({\mathbb D}_A(\varepsilon))$ is an open neighborhood of $\Dd_{A,\xi}$ in $\R^3$. 

\begin{defi}\label{def1}\em
A (continuous) path $\alpha:[a,b]\rightarrow\R^3$ \emph{meets transversally once the warped disc $\Dd_{A,\xi}$} if there exist $s_0\in(a,b)$ and $\varepsilon>0$ such that $J:=\alpha^{-1}(\Dd_{A,\xi}(\varepsilon))=(s_0-\varepsilon,s_0+\varepsilon)$ is an open subinterval of $[a,b]$ and $(\zeta\circ\alpha)|_J(t)=(0,0,t-s_0)$.
\end{defi}

\begin{remark}\label{rem0}\em
If the path $\alpha:[a,b]\to\R^3$ meets transversally once the warped disc $\Dd_{A,\xi}$, then $\alpha([a,b])\cap\partial\Dd_{A,\xi}=\varnothing$.
\end{remark}

Let $C$ be a topological space homeomorphic to a closed disc and let $\phi:C\to\R^3$ be a continuous map. The restriction $\partial\phi:=\phi|_{\partial C}$ is called the \em boundary map \em of $\phi$. We say that the boundary map $\partial\phi$ \em meets transversally once \em a warped disc $\Dd_{A,\xi}\subset\R^3$ if there exists a parameterization $\beta$ of $\partial C$ such that $\alpha:=\phi\circ\beta$ meets transversally once the warped disc $\Dd_{A,\xi}$.

Given a path-connected topological space $X$ and a point $x_0\in X$ we denote the fundamental group of $X$ at the base point $x_0$ with $\pi_1(X,x_0)$. Each path $\alpha$ starting and ending at $x_0$ is called a loop with base point $x_0$ and represents an element of $\pi_1(X,x_0)$, that we denote with $[\alpha]$.

\begin{lem}\label{lem2}
Let $\Dd_{A,\xi}$ be a warped disc of $\R^3$ and let $X:=\R^3\setminus\partial\Dd_{A,\xi}$. Let $\alpha:[a,b]\to X$ be a loop with base point $x_0\in X$ that meets transversally once $\Dd_{A,\xi}$. Then $[\alpha]$ is a generator of $\pi_1(X,x_0)\cong{\mathbb Z}$.
\end{lem}
\begin{proof} 
Keep the notations introduced above. Let $s_0\in(a,b)$ and $\varepsilon>0$ be such that 
$$
J:=\alpha^{-1}(\Dd_{A,\xi}(\varepsilon))=(s_0-\varepsilon,s_0+\varepsilon)
$$ 
is an open subinterval of $[a,b]$ and $(\zeta\circ\alpha)|_J(t)=(0,0,t-s_0)$. After a reparameterization of $\alpha$ we may assume $s_0=0$.

As $\zeta$ is a homeomorphism of $\R^3$, we will prove the statement for $\beta:=\zeta\circ\alpha$, $Y:=\R^3\setminus\partial{\mathbb D}_A$ and the base point $y_0:=\beta(-\varepsilon)=(0,0,-\varepsilon)$. Consider the path $\gamma:[0,1]\rightarrow\R^3$ given by
$$
\gamma(t):=
\begin{cases}
(3(A+\varepsilon)t,0,\varepsilon)&\text{if $0\le t\le\frac{1}{3}$,}\\
(A+\varepsilon, 0, \varepsilon-(t-\frac{1}{3})6\varepsilon)&\text{if $\frac{1}{3}<t\le \frac{2}{3}$,}\\
(A+\varepsilon-3(A+\varepsilon)(t-\frac{2}{3}), 0, -\varepsilon)&\text{if $\frac{2}{3}<t\le 1$}.
\end{cases}
$$

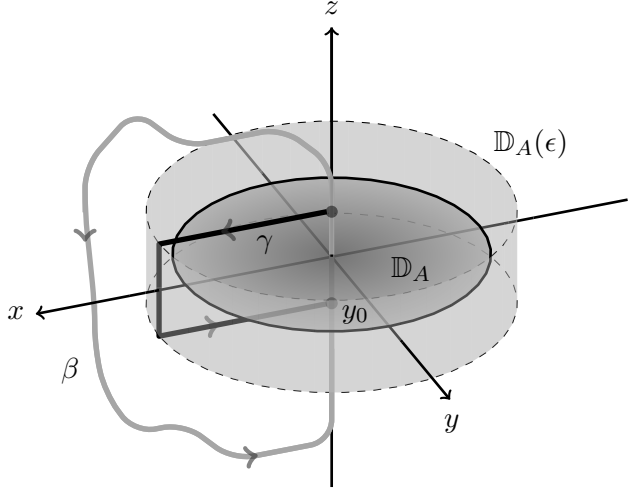
\begin{figure}[ht!]
\begin{center}
\begin{tikzpicture}[scale=0.8,line join=round,decoration={
    markings,
    mark=at position 0.5 with {\arrow[black!80,opacity=0.6]{>}}}]
\tikzstyle{estiloejes} = [line width=1pt]
\tikzstyle{estilodiscos} = [shading=radial]
\tikzstyle{estiloborde1} = [line width=2pt, draw=gray!70, rounded corners=8pt]
\tikzstyle{estiloborde2} = [line width=2pt, decorate]
\tikzstyle{estiloborde} = [line width=2pt]
\draw[estiloejes](-1.164,1.416)--(-1.947,2.368);
\filldraw[fill=gray!50,fill opacity=0.4,draw=none](-.677,.688)--(-.677,2.214)--(-.203,2.248)--(-.203,.722)--cycle;
\filldraw[fill=gray!50,fill opacity=0.4,draw=none](-.203,.722)--(-.203,2.248)--(.277,2.245)--(.277,.719)--cycle;
\draw[estiloejes](0,-3.816)--(0,3.8);
\filldraw[fill=gray!50,fill opacity=0.4,draw=black, dashed](-2.839,-1.316)--(-2.982,-1.093)--(-3.051,-.862)--(-3.045,-.628)--(-2.964,-.398)--(-2.81,-.177)--(-2.587,.03)--(-2.301,.217)--(-1.957,.38)--(-1.566,.515)--(-1.136,.618)--(-.677,.688)--(-.203,.722)--(.277,.719)--(.75,.679)--(1.204,.605)--(1.629,.496)--(2.014,.356)--(2.349,.189)--(2.626,-.001)--(2.839,-.211)--(2.982,-.433)--(3.051,-.665)--(3.045,-.898)--(2.964,-1.128)--(2.81,-1.349)--(2.587,-1.556)--(2.301,-1.743)--(1.957,-1.906)--(1.566,-2.041)--(1.136,-2.145)--(.677,-2.214)--(.203,-2.248)--(-.277,-2.245)--(-.75,-2.206)--(-1.204,-2.131)--(-1.629,-2.022)--(-2.014,-1.883)--(-2.349,-1.716)--(-2.626,-1.525)--cycle;
\filldraw[fill=gray!50,fill opacity=0.4,draw=none](.277,.719)--(.277,2.245)--(.75,2.206)--(.75,.679)--cycle;
\filldraw[fill=gray!50,fill opacity=0.4,draw=none](-1.136,.618)--(-1.136,2.145)--(-.677,2.214)--(-.677,.688)--cycle;
\filldraw[fill=gray!50,fill opacity=0.4,draw=none](.75,.679)--(.75,2.206)--(1.204,2.131)--(1.204,.605)--cycle;
\draw[estiloejes](-1.136,1.382)--(-1.164,1.416);
\filldraw[fill=gray!50,fill opacity=0.4,draw=none](-1.566,.515)--(-1.566,2.041)--(-1.136,2.145)--(-1.136,.618)--cycle;
\filldraw[fill=gray!50,fill opacity=0.4,draw=none](1.204,.605)--(1.204,2.131)--(1.629,2.022)--(1.629,.496)--cycle;
\filldraw[fill=gray!50,fill opacity=0.4,draw=none](-1.957,.38)--(-1.957,1.906)--(-1.566,2.041)--(-1.566,.515)--cycle;
\filldraw[fill=gray!50,fill opacity=0.4,draw=none](1.629,.496)--(1.629,2.022)--(2.014,1.883)--(2.014,.356)--cycle;
\draw[estiloejes](1.136,-1.382)--(-1.136,1.382);
\filldraw[fill=gray!50,fill opacity=0.4,draw=none](-2.301,.217)--(-2.301,1.743)--(-1.957,1.906)--(-1.957,.38)--cycle;
\filldraw[fill=gray!50,fill opacity=0.4,draw=none](2.014,.356)--(2.014,1.883)--(2.349,1.716)--(2.349,.189)--cycle;
\filldraw[fill=gray!50,fill opacity=0.4,draw=none](-2.587,.03)--(-2.587,1.556)--(-2.301,1.743)--(-2.301,.217)--cycle;
\filldraw[fill=gray!50,fill opacity=0.4,draw=none](2.349,.189)--(2.349,1.716)--(2.626,1.525)--(2.626,-.001)--cycle;
\filldraw[fill=gray!50,fill opacity=0.4,draw=none](-2.81,-.177)--(-2.81,1.349)--(-2.587,1.556)--(-2.587,.03)--cycle;
\filldraw[fill=gray!50,fill opacity=0.4,draw=none](2.626,-.001)--(2.626,1.525)--(2.839,1.316)--(2.839,-.211)--cycle;
\filldraw[fill=gray!50,fill opacity=0.4,draw=none](-3.045,-.628)--(-3.045,.898)--(-2.964,1.128)--(-2.964,-.398)--cycle;
\filldraw[fill=gray!50,fill opacity=0.4,draw=none](-3.051,-.862)--(-3.051,.665)--(-3.045,.898)--(-3.045,-.628)--cycle;
\draw[estiloejes](2.839,.553)--(-2.839,-.553);
\draw[estiloejes](0,-.763)--(0,.038);
\filldraw[fill=gray!50,fill opacity=0.4,draw=black,dashed](-2.626,.001)--(-2.349,-.189)--(-2.014,-.356)--(-1.629,-.496)--(-1.204,-.605)--(-.75,-.679)--(-.277,-.719)--(.203,-.722)--(.677,-.688)--(1.136,-.618)--(1.566,-.515)--(1.957,-.38)--(2.301,-.217)--(2.587,-.03)--(2.81,.177)--(2.964,.398)--(3.045,.628)--(3.051,.862)--(2.982,1.093)--(2.839,1.316)--(2.626,1.525)--(2.349,1.716)--(2.014,1.883)--(1.629,2.022)--(1.204,2.131)--(.75,2.206)--(.277,2.245)--(-.203,2.248)--(-.677,2.214)--(-1.136,2.145)--(-1.566,2.041)--(-1.957,1.906)--(-2.301,1.743)--(-2.587,1.556)--(-2.81,1.349)--(-2.964,1.128)--(-3.045,.898)--(-3.051,.665)--(-2.982,.433)--(-2.839,.211)--cycle;
\draw[estiloborde2](-2.839,-.514)--(-2.839,-1.316)--(-.081,-.779);
\draw[estiloborde](-2.839,-.514)--(-2.839,-1.316)--(-.081,-.779);
\draw[estiloborde1, opacity=1](-3.9,-.789)--(-3.893,-1.521)--(-3.731,-2.253)--(-3.244,-2.921)--(-2.433,-2.763)--(-1.622,-3.368)--(-.811,-3.211)--(0,-3.053)--(0,-1.526)--(0,-.763)--(0,.038)--(0,.763)--(0,1.526)--(-.811,2.132)--(-2.433,1.816)--(-3.244,2.421)--(-4.056,1.5)--(-4.056,.737)--cycle;
\filldraw[fill opacity=0.5,line width=1pt,estilodiscos](-.973,1.222)--(-1.342,1.134)--(-1.678,1.018)--(-1.972,.878)--(-2.218,.718)--(-2.409,.541)--(-2.541,.351)--(-2.61,.154)--(-2.615,-.046)--(-2.556,-.244)--(-2.433,-.436)--(-2.251,-.615)--(-2.013,-.778)--(-1.726,-.922)--(-1.396,-1.041)--(-1.032,-1.134)--(-.643,-1.198)--(-.237,-1.232)--(.174,-1.234)--(.581,-1.206)--(.973,-1.146)--(1.342,-1.057)--(1.678,-.942)--(1.972,-.802)--(2.218,-.641)--(2.409,-.464)--(2.541,-.275)--(2.61,-.077)--(2.615,.123)--(2.556,.321)--(2.433,.512)--(2.251,.691)--(2.013,.855)--(1.726,.998)--(1.396,1.117)--(1.032,1.21)--(.643,1.275)--(.237,1.308)--(-.174,1.311)--(-.581,1.282)--cycle;
\draw[estiloejes,->](0,0)--(0,3.8);
\draw[estiloborde1,opacity=0.9](-3.9,-.789)--(-3.893,-1.521)--(-3.731,-2.253)--(-3.244,-2.921)--(-2.433,-2.763)--(-1.622,-3.368)--(-.811,-3.211)--(0,-3.053)--(0,-1.526)(0,.038)--(0,.763)--(0,1.526)--(-.811,2.132)--(-2.433,1.816)--(-3.244,2.421)--(-4.056,1.5)--(-4.056,.737);
\draw[estiloborde1,decorate]
(-4.056,1.)--(-4.056,-0.5);
\draw[estiloborde1,decorate]
(-1.622,-3.368)--(-.811,-3.211);
\filldraw[fill=gray!50,fill opacity=0.4,draw=none](-2.964,-.398)--(-2.964,1.128)--(-2.81,1.349)--(-2.81,-.177)--cycle;
\draw[estiloejes](3.319,.646)--(2.911,.567);
\draw[estiloejes](2.911,.567)--(2.839,.553);
\filldraw[fill=gray!50,fill opacity=0.4,draw=none](2.839,-.211)--(2.839,1.316)--(2.982,1.093)--(2.982,-.433)--cycle;
\filldraw[fill=gray!50,fill opacity=0.4,draw=none](2.982,-.433)--(2.982,1.093)--(3.051,.862)--(3.051,-.665)--cycle;
\draw[estiloborde](-.081,.747)--(-2.839,.211)--(-2.839,-.514);
\draw[estiloborde2](-.081,.747)--(-2.839,.211)--(-2.839,-.514);
\draw[fill=black, draw=none, fill opacity=0.5] (0,.747) circle (0.1cm);
\draw[fill=black, draw=none, fill opacity=0.3] (0,-.77) circle (0.1cm);
\draw[estiloejes](4.867,.947)--(3.319,.646);
\filldraw[fill=gray!50,fill opacity=0.4,draw=none](3.051,-.665)--(3.051,.862)--(3.045,.628)--(3.045,-.898)--cycle;
\filldraw[fill=gray!50,fill opacity=0.4,draw=none](-2.982,-1.093)--(-2.982,.433)--(-3.051,.665)--(-3.051,-.862)--cycle;
\filldraw[fill=gray!50,fill opacity=0.4,draw=none](3.045,-.898)--(3.045,.628)--(2.964,.398)--(2.964,-1.128)--cycle;
\filldraw[fill=gray!50,fill opacity=0.4,draw=none](-2.839,-1.316)--(-2.839,.211)--(-2.982,.433)--(-2.982,-1.093)--cycle;
\filldraw[fill=gray!50,fill opacity=0.4,draw=none](2.964,-1.128)--(2.964,.398)--(2.81,.177)--(2.81,-1.349)--cycle;
\filldraw[fill=gray!50,fill opacity=0.4,draw=none](-2.626,-1.525)--(-2.626,.001)--(-2.839,.211)--(-2.839,-1.316)--cycle;
\filldraw[fill=gray!50,fill opacity=0.4,draw=none](2.81,-1.349)--(2.81,.177)--(2.587,-.03)--(2.587,-1.556)--cycle;
\filldraw[fill=gray!50,fill opacity=0.4,draw=none](-2.349,-1.716)--(-2.349,-.189)--(-2.626,.001)--(-2.626,-1.525)--cycle;
\draw[estiloejes](-2.839,-.553)--(-2.911,-.567);
\filldraw[fill=gray!50,fill opacity=0.4,draw=none](2.587,-1.556)--(2.587,-.03)--(2.301,-.217)--(2.301,-1.743)--cycle;
\draw[estiloejes](-2.911,-.567)--(-3.063,-.596);
\draw[estiloejes](-3.063,-.596)--(-3.319,-.646);
\draw[estiloejes,->](-3.319,-.646)--(-4.867,-.947);
\filldraw[fill=gray!50,fill opacity=0.4,draw=none](-2.014,-1.883)--(-2.014,-.356)--(-2.349,-.189)--(-2.349,-1.716)--cycle;
\filldraw[fill=gray!50,fill opacity=0.4,draw=none](2.301,-1.743)--(2.301,-.217)--(1.957,-.38)--(1.957,-1.906)--cycle;
\filldraw[fill=gray!50,fill opacity=0.4,draw=none](-1.629,-2.022)--(-1.629,-.496)--(-2.014,-.356)--(-2.014,-1.883)--cycle;
\filldraw[fill=gray!50,fill opacity=0.4,draw=none](1.957,-1.906)--(1.957,-.38)--(1.566,-.515)--(1.566,-2.041)--cycle;
\filldraw[fill=gray!50,fill opacity=0.4,draw=none](-1.204,-2.131)--(-1.204,-.605)--(-1.629,-.496)--(-1.629,-2.022)--cycle;
\filldraw[fill=gray!50,fill opacity=0.4,draw=none](1.566,-2.041)--(1.566,-.515)--(1.136,-.618)--(1.136,-2.145)--cycle;
\filldraw[fill=gray!50,fill opacity=0.4,draw=none](-.75,-2.206)--(-.75,-.679)--(-1.204,-.605)--(-1.204,-2.131)--cycle;
\filldraw[fill=gray!50,fill opacity=0.4,draw=none](1.136,-2.145)--(1.136,-.618)--(.677,-.688)--(.677,-2.214)--cycle;
\filldraw[fill=gray!50,fill opacity=0.4,draw=none](-.277,-2.245)--(-.277,-.719)--(-.75,-.679)--(-.75,-2.206)--cycle;
\filldraw[fill=gray!50,fill opacity=0.4,draw=none](.677,-2.214)--(.677,-.688)--(.203,-.722)--(.203,-2.248)--cycle;
\filldraw[fill=gray!50,fill opacity=0.4,draw=none](.203,-2.248)--(.203,-.722)--(-.277,-.719)--(-.277,-2.245)--cycle;
\draw[estiloejes](1.164,-1.416)--(1.136,-1.382);
\draw[estiloejes,<-](1.947,-2.368)--(1.164,-1.416);
\path (2.271,-2.763) node[left]{$y$} (-5.5,-1.2) node[above right]{$x$} (0,4.4) node[below]{$z$};\path (.811,-.2) node[right]{${\mathbb D}_A$} (0,-0.95) node[right]{$y_0$};
\path (2.5,1.468) node[above right]{${\mathbb D}_A(\epsilon)$};
\path (-4.3,-1.5) node[below]{$\beta$};
\path (-1.1,.526) node[below]{$\gamma$};
\end{tikzpicture}
\caption{The path $\beta$ meets transversally once the disk ${\mathbb D}_A$.}\label{pic:disco}
\end{center}
\end{figure}
Write $\beta_0:=\beta|_{J}$ and $\beta_1:=\beta|_{[\varepsilon,b]}\ast\beta|_{[a,-\varepsilon]}$. We claim:\em
$$
\quad[\beta]=[\beta_0\ast\beta_1]=[\beta_0\ast\gamma]\,\cdot\,[\gamma^{-1}\ast\beta_1]=g\,\cdot\,e=g,
$$
where $e$ and $g$ are respectively the identity element and a generator of $\pi_1(Y,y_0)\cong{\mathbb Z}$\em.

The loop $\gamma^{-1}\ast\beta_1$ with base point $y_0$ is contained in $\R^3\setminus{\mathbb D}_A$, which is a simply connected space. Consequently, $[\gamma^{-1}*\beta_1]=e$ in $\pi_1(Y,y_0)$. 

The class $[\beta_0\ast\gamma]$ generates $\pi_1(Y,y_0)$. Indeed, $Y$ has as deformation retract the set $Z:=\partial{\mathbb D}_A(\varepsilon)\cup I_\varepsilon$ where $I_\varepsilon:=\{(0,0)\}\times\{-\varepsilon\le z\le \varepsilon\}$. It is an exercise of algebraic topology to show that $[\beta_0\ast\gamma]$ is a generator of $\pi_1(Z,y_0)\cong\pi_1(Y,y_0)\cong{\mathbb Z}$, as required.\qed 
\end{proof}

\begin{lem}\label{lem3}
Let $\phi:C\rightarrow X$ be a continuous map and assume that $C$ is homeomorphic to a closed disc. Let $\beta:[a,b]\to\partial C$ be a parameterization starting and ending at $z_0\in\partial C$. Then $[\phi\circ\beta]$ is the identity element of $\pi_1(X,\phi(z_0))$.
\end{lem}
\begin{proof} 
Let $\psi:C\to\{x^2+y^2\leq 1\}$ be a homeomorphism. The continuous map 
$$
H:[0,1]\times[a,b]\to X,\ (\rho,t)\mapsto(\phi\circ\psi^{-1})(\rho\cdot(\psi\circ\beta)(t)+(1-\rho)\cdot\psi(z_0)))
$$
is a homotopy map between $\phi\circ\beta$ and the constant path, as required.\qed
\end{proof}

\begin{prop}\label{prop:top}
Let $C$ be a topological space homeomorphic to a closed disc and $\phi:C\to\R^3$ a continuous map. Assume $\partial\phi:\partial C\to\R^3$ meets transversally once a warped disc $\Dd\subset\R^3$. Then $\partial\Dd\cap\phi(\Int(C))\neq\varnothing$.
\end{prop}
\begin{proof}
Assume by contradiction $\partial\Dd\cap\phi(\Int(C))=\varnothing$. As $\partial\phi$ meets transversally once $\Dd$, the image $\phi(\partial C)$ does not intersect $\partial\Dd$ by Remark~\ref{rem0}. Thus, $\phi(C)\subset X:=\R^3\setminus\partial\Dd$. Let $\beta:[a,b]\to\partial C$ be a parameterization starting and ending at $z_0\in\partial C$ such that $\phi\circ\beta$ meets transversally once $\Dd$. By Lemma~\ref{lem3} the class $[\phi\circ\beta]$ is the identity element of $\pi_1(X,\phi(z_0))$. However, by Lemma~\ref{lem2} the class $[\phi\circ\beta]$ is a generator of $\pi_1(X,\phi(z_0))\cong{\mathbb Z}$, which is a contradiction. Consequently, $\partial\Dd\cap\phi(\Int(C))\neq\varnothing$, as required.\qed
\end{proof}

\section{Proof of Theorem \ref{main}}\label{s3}

Observe first that the map $f$ in the statement of Theorem \ref{main} is the composition $f_2\circ f_1$ of the polynomial maps
\begin{align*}
&f_1:\R^2\to\R^2,\quad (x,y)\mapsto(x^2,y^2),\\
&f_2:\R^2\to\R^2,\quad (x,y)\mapsto((xy^2+x^2y-y-1)^2+x^3y^2,(x^3y+xy-x-1)^2+x^3y^2).
\end{align*}
As $f_1(\R^2)$ is the closed quadrant $\overline{\Qq}:=\{x\geq0,y\geq0\}$, we have to prove the equality
\begin{equation}
f_2(\overline{\Qq})=\Qq.
\end{equation}
The inclusion $f_2(\overline{\Qq})\subset\Qq$ is straightforward because both components of $f_2$ are strictly positive on $\overline{\Qq}$. It only remains to show the inclusion 
\begin{equation}\label{principal}
\Qq\subset f_2(\overline{\Qq}).
\end{equation}

\subsection{Reduction of the proof of inclusion~\em\eqref{principal}}
Consider the (continuous) semialgebraic maps 
\begin{align*}
&g:\overline{\Qq}\rightarrow\R^3,\ (x,y)\mapsto(xy^2+x^2y-y-1,x^{3/2}y,x^3y+xy-x-1)\\
&h:\R^3\rightarrow\R^2,\ (x,y,z)\mapsto(x^2+y^2,y^2+z^2). 
\end{align*}
As $f_2=h\circ g$, we have to show that for each tuple $(A^2,B^2)\in\Qq$ there exists $(x_0,y_0)\in\overline{\Qq}$ such that $(h\circ g)(x_0,y_0)=(A^2,B^2)$. This is equivalent to check that the intersection $h^{-1}(\{(A^2,B^2)\})\cap g(\overline{\Qq})$ is non-empty. 

Denote $\Ss:=g(\overline{\Qq})$ and fix values $B\ge A>0$. It holds that sets
\begin{align*}
&h^{-1}(\{(A^2,B^2)\})=\{x^2+y^2=A^2,y^2+z^2=B^2\},\\
&h^{-1}(\{(B^2,A^2)\})=\{y^2+z^2=A^2,x^2+y^2=B^2\}
\end{align*} 
contain respectively the boundaries of the warped discs
\begin{align}
\Dd_1:\,z=\xi_1(x,y),\quad x^2+y^2\le A^2,\label{d1}\\
\Dd_2:\,x=\xi_2(y,z),\quad y^2+z^2\le A^2,\label{d2}
\end{align}
for the (continuous) semialgebraic functions 
\begin{align}
&\xi_1:\R^2\to\R,\ (x,y)\mapsto\sqrt{B^2-\min\{y^2,B^2\}},\label{xi1}\\
&\xi_2:\R^2\to\R,\ (y,z)\mapsto\sqrt{B^2-\min\{y^2,B^2\}}.\label{xi2}
\end{align}

Consequently, we are reduced to prove:

\paragraph{}\hspace{-6mm}\label{311} {\em For fixed values $B\ge A>0$ the intersections $\partial\Dd_1\cap\Ss$ and $\partial\Dd_2\cap\Ss$ are non-empty.}

\subsection{Proof of Statement \ref{311}}

Write $\Rr:=[0,+\infty)\times(0,\tfrac{\text{\textpi}}{2})$ and $\overline{\Rr}:=[0,+\infty)\times[0,\tfrac{\text{\textpi}}{2}]$. Consider the map $\phi:=(\phi_1,\phi_2,\phi_3):\R^2\to\R^3$ where\begin{align*}
&\phi_1(\rho,\theta):=\cos\theta\sin\theta(\cos\theta-
\sin\theta)^2\\
&\hspace{3.2cm}+\rho(2\cos^4\theta\sin\theta+\cos\theta\sin^4\theta+\cos^5\theta)+\rho^2\cos^5\theta\sin\theta,\\
&\phi_2(\rho,\theta):=\sqrt{\cos\theta\sin\theta}(\cos\theta+\sin\theta+
\rho\cos\theta\sin\theta),\\
&\phi_3(\rho,\theta):=\rho\sin\theta.
\end{align*}
Let us prove now some properties of the map $\phi$ and the sets $\Rr$ and $\overline{\Rr}$:\setcounter{paragraph}{0}

\paragraph{}\label{lem:uno}\hspace{-6mm} $\phi(\Rr)\subset\Ss$.

\vspace{2mm}
\noindent\begin{proof}
The analytic map 
$$
\psi:\Rr\to\Qq,\quad (\rho,\theta)\mapsto \left(\frac{\sin\theta}{\cos\theta},\frac{(\cos\theta+\sin\theta+\rho\cos\theta
\sin\theta)\cos^2\theta}{\sin\theta}\right),
$$
satisfies $\psi(\Rr)\subset\overline{\Qq}$ and $g\circ\psi=\phi|_{\Rr}$. Consequently, $\phi(\Rr)\subset\Ss$, as required.\qed
\end{proof}

\paragraph{}\hspace{-5mm}\label{lem:tres}
\em The inequality {$\phi_1^2(\rho,\theta)+\phi_3^2(\rho,\theta)\ge\frac{\rho^2}{4}$ holds for each $(\rho,\theta)\in\overline{\Rr}$.} Consequently,
\begin{equation}
\dist(\phi(\rho,\theta),{\bf0})\geq\frac{\rho}{2}\label{eq:1}
\end{equation}
for each $(\rho,\theta)\in\overline{\Rr}$. \em

\vspace{2mm}
\noindent\begin{proof}
As $\rho$, $\cos\theta$, $\sin\theta$ are $\geq0$ on $\overline{\Rr}$, we have
\begin{equation*}
\begin{split}
\phi_1(\rho,\theta)\ge\,&\rho\cos\theta(\cos^4\theta+\sin^4\theta)=\rho\cos\theta(1-2\cos^2\theta\sin^2\theta)\\
=\,&\rho\cos\theta\left(1-\frac{\sin^2(2\theta)}{2}\right)
\ge \frac{\rho}{2}\cos\theta.
\end{split}
\end{equation*}
In addition, $\phi_3(\rho,\theta)=\rho\sin\theta\ge\frac{\rho}{2}\sin\theta$, so
$$
\phi_1^2(\rho,\theta)+\phi_3^2(\rho,\theta)\ge \frac{\rho^2}{4}\cos^2\theta+\frac{\rho^2}{4}\sin^2\theta=\frac{\rho^2}{4},
$$
as required.\qed
\end{proof}

\paragraph{}\hspace{-6mm}\label{lem:dos0} The map $\phi$ satisfies $\phi(0,\theta)=\phi(0,\tfrac{\text{\textpi}}{2}-\theta)$ for $\theta\in[0,\tfrac{\text{\textpi}}{2}]$. Fix $M>0$ and consider the rectangle $\overline{\Rr}_M:=[0,M]\times[0,\tfrac{\text{\textpi}}{2}]$. Denote $\phi_M:=\phi|_{\overline{\Rr}_M}$. Identify the points $(0,\theta)$ and $(0,\frac{\text{\textpi}}{2}-\theta)$ for $\theta\in[0,\tfrac{\text{\textpi}}{2}]$ and endow the quotient space $\tilde{\overline{\Rr}}_M$ with the quotient topology. Observe that the interior $\Int(\tilde{\overline{\Rr}}_M)$ of $\tilde{\overline{\Rr}}_M$ as a topological manifold with boundary is the quotient space $\tilde{\Rr}_M$ obtained identifying the points $(0,\theta)$ and $(0,\frac{\text{\textpi}}{2}-\theta)$ of $\Rr_M:=[0,M)\times(0,\tfrac{\text{\textpi}}{2})$, where $\theta\in(0,\tfrac{\text{\textpi}}{2})$.


\begin{figure}[!ht]
\begin{center}
\begin{tabular}{cc}
\hskip 1.5cm\resizebox{6.2cm}{5cm}{\begin{tikzpicture}[trim left=0cm]
\pgfplotsset{ticks=none}
\begin{axis}[point meta min=-4, point meta max=3,view={-45}{30}, scale=3, axis lines=middle, footnotesize,xlabel=$$,ylabel=$$,
xmin=-0.5,xmax=2.5,ymin=-0.3,ymax=1.5,zmin=-0.5,zmax=2.5,
unit vector ratio=]
\addplot3[opacity=1, fill opacity=1, surf,samples=20,colormap/blackwhite,
variable=\u, variable y=\v,
domain=0:2, y domain=0:90]
({min(cos(v)*sin(v)*(sin(v)-cos(v))^2+
u*(2*(cos(v))^4*sin(v)+cos(v)*(sin(v))^4+(cos(v))^5)
+u^2*(cos(v))^5*sin(v),10)}, {min(sqrt(sin(v)*cos(v))*(sin(v)+cos(v)
+u*cos(v)*sin(v)),10)}, {min(u*sin(v),10)});
\addplot3[opacity=0.0, fill opacity=0.0,surf,gray,z buffer=sort,samples=40,
variable=\u, variable y=\v,
domain=-1.5:2.5, y domain=0:360]
({3*cos(v)}, {3*sin(v)}, {u});
\addplot3[opacity=1, fill opacity=0.0,domain=0:90,samples=40,samples y=0]
({min(cos(x)*sin(x)*(sin(x)-cos(x))^2,2)}, {min(sqrt(sin(x)*cos(x))*(sin(x)+cos(x)),2)}, {0});
\end{axis}
\end{tikzpicture}}& \hskip 0.2cm
\resizebox{7.6cm}{5cm}{
\begin{tikzpicture}
\pgfplotsset{ticks=none}
\begin{axis}[point meta min=-4, point meta max=3,view={45}{30}, scale=3.5, axis lines=middle, footnotesize,xlabel=$$,ylabel=$$,
xmin=-0.5,xmax=3.5,ymin=-0.3,ymax=1.5,zmin=-0.5,zmax=2.5,
unit vector ratio=]
\addplot3[opacity=1, fill opacity=1, surf,samples=20,colormap/blackwhite,
variable=\u, variable y=\v,
domain=0:2, y domain=0:90]
({min(cos(v)*sin(v)*(sin(v)-cos(v))^2+
u*(2*(cos(v))^4*sin(v)+cos(v)*(sin(v))^4+(cos(v))^5)
+u^2*(cos(v))^5*sin(v),10)}, {min(sqrt(sin(v)*cos(v))*(sin(v)+cos(v)
+u*cos(v)*sin(v)),10)}, {min(u*sin(v),10)});
\addplot3[opacity=0.0, fill opacity=0.0,surf,gray,z buffer=sort,samples=40,
variable=\u, variable y=\v,
domain=-1.5:2.5, y domain=0:360]
({3*cos(v)}, {3*sin(v)}, {u});
\addplot3[opacity=1, fill opacity=0.0,domain=0:90,samples=40,samples y=0]
({min(cos(x)*sin(x)*(sin(x)-cos(x))^2,2)}, {min(sqrt(sin(x)*cos(x))*(sin(x)+cos(x)),2)}, {0});
\end{axis}
\end{tikzpicture}
}
\end{tabular}
\caption{Left and right views of $\phi_M(\Rr_M)\subset\Ss$.}\label{pic:surface1}
\end{center}
\end{figure}


The canonical projection $\proy_M:\overline{\Rr}_M\to\tilde{\overline{\Rr}}_M$ is continuous. As $\phi_M$ is compatible with $\proy_M$, there exists a continuous map $\tilde{\phi}_M:\tilde{\Rr}_M\to\R^3$ such that the following diagram is commutative. In addition, $\tilde{\phi}_M(\tilde{\Rr}_M)=\phi(\Rr_M)\subset\Ss$.
\begin{center}
\begin{tikzpicture}[scale=1.5]
\node (A0) at (0,1) {$\Rr_M$};
\node (A) at (1,1) {$\overline{\Rr}_M$};
\node (C0) at (0,0) {$\tilde{\Rr}_M$};
\node (C) at (1,0) {$\tilde{\overline{\Rr}}_M$};
\node (D) at (2,0) {$\R^3$};

\path[->,font=\scriptsize,>=angle 90]
(A) edge node[above right]{$\phi_M$} (D)
(A) edge node[left]{$\proy_M$} (C)
(A0) edge node[left]{$\proy_M|_{\Rr_M}$} (C0)
(C) edge node[below]{$\tilde{\phi}_M$} (D);

\path[right hook->,font=\scriptsize,>=angle 90]
(A0) edge node[left]{} (A)
(C0) edge node[left]{} (C);

\end{tikzpicture}
\end{center}

\paragraph{}\hspace{-5mm}\label{lem:dos} {\em $\tilde{\overline{\Rr}}_M$ is homeomorphic to a disc and its boundary is the set
$$
\proy_M(\{\rho=M\}\cup\{\theta=0\}\cup\{\theta=\tfrac{\text{\textpi}}{2}\}).
$$}
\noindent\begin{proof}
Identify $\R^2$ with $\C$ (interchanging the order of the variables $(\rho,\theta)\leadsto(\theta,\rho)$) and consider the continuous map 
$$
\mu:\C\to\C,\ z:=\theta+\sqrt{-1}\rho\mapsto w:=u+\sqrt{-1}v=(\tfrac{4}{\text{\textpi}}z-1)^2.
$$
The restriction $\mu|_{\{\rho>0\}}:\{\rho>0\}\to\C\setminus([0,+\infty)\times\{0\})$ is a homeomorphism and the image of $\overline{\Rr}_M\setminus\{\rho=0\}$ is
$$
\Tt_M:=\big\{(u,v)\in\R^2:\ (\tfrac{\pi v}{8M})^2-(\tfrac{4M}{\text{\textpi}})^2\leq u\leq1-(\tfrac{v}{2})^2\big\}\setminus([0,1]\times\{0\}).
$$
The closure $\overline{\Tt}_M$ of $\Tt_M$ is a compact convex set (as it is a closed bounded intersection of two convex sets). By \cite[Cor.11.3.4]{ber1} $\overline{\Tt}_M$ is homeomorphic to a closed disc. In addition
$$
\mu|_{\{\rho=0\}}:\{\rho=0\}\to[0,+\infty)\times\{0\},\ \theta\mapsto(\tfrac{4}{\text{\textpi}}\theta-1)^2
$$
transforms the segment $[0,\tfrac{\text{\textpi}}{2}]\times\{0\}$ onto the interval $[0,1]$. The preimage of $t_0\in[0,1]$ under $\mu|_{\{\rho=0\}}$ is 
$$
\{\theta_1:=\tfrac{\text{\textpi}}{4}(1+\sqrt{t_0}),\theta_2:=\tfrac{\text{\textpi}}{4}(1-\sqrt{t_0})\}.
$$

\begin{figure}[!ht]
\begin{center}
\begin{tikzpicture}[scale=0.99]

\begin{scope}[scale=0.5,xshift=4cm, yshift=5cm]
\draw[line width=1.5pt,rotate=-90,draw,fill=gray!40] (6,11) parabola bend (0,20) (-6,11) parabola bend (0,10) (6,11);
\draw[->, line width=1pt] (9,0) -- (21,0);
\draw[->, line width=1pt] (19,-6.5) -- (19,6.5);

\draw (20.5,-0.5) node{$1$};
\draw (18.5,-0.5) node{$0$};
\draw (14,2.5) node{$\overline{\Tt}_M$};
\draw[white,line width=1.5pt,dotted] (19,0) -- (20,0);
\draw (19,0) node{\large$\bullet$};
\draw (20,0) node{\large$\bullet$};
\end{scope}

\draw[draw=none, fill=gray!40] (0.5,0) to (3.5,0) to (3.5,5) to (0.5,5) to (0.5,0);
\draw[line width=1.5pt] (3.5,0) to (3.5,5) to (0.5,5) to (0.5,0);
\draw[->, line width=1pt] (-0.5,0) -- (4.5,0);
\draw[->, line width=1pt] (0.5,-1) -- (0.5,6);
\draw[white,line width=1.5pt,dotted] (0.5,0) to (3.5,0);
\draw[line width=1.5pt] (2,-0.25) -- (2,0.25);
\draw [line width=1.5pt, ->] (4.5,2.5)--(5.25,2.5) node [above] {$\mu$} -- (6,2.5);
\draw (0.25,-0.25) node{$0$};
\draw (3.5,-0.25) node{$\tfrac{\text{\textpi}}{2}$};
\draw (4.4,-0.25) node{$\theta$};
\draw (2.25,-0.25) node{$\tfrac{\text{\textpi}}{4}$};
\draw (0.25,5) node{$M$};
\draw (2.25,2.5) node{$\overline{\Rr}_M$};
\draw (0.25,2.5) node{$\rho$};
\end{tikzpicture}
\end{center}
\caption{Behavior of the map $\mu:\overline{\Rr}_M\to\overline{\Tt}_M$.}
\end{figure}

As $\theta_1=\tfrac{\text{\textpi}}{2}-\theta_2$, the map $\lambda:=\mu|_{\overline{\Rr}_M}:\overline{\Rr}_M\to\overline{\Tt}_M$ factors through $\tilde{\overline{\Rr}}_M$ and there exists a continuous map $\tilde{\lambda}:\tilde{\overline{\Rr}}_M\to\overline{\Tt}_M$ such that the following diagram is commutative.
\begin{center}
\begin{tikzpicture}[scale=1.5]
\node (A0) at (0,1) {$\Rr_M$};
\node (A) at (1,1) {$\overline{\Rr}_M$};
\node (C0) at (0,0) {$\tilde{\Rr}_M$};
\node (C) at (1,0) {$\tilde{\overline{\Rr}}_M$};
\node (D) at (2,0) {$\Ss_M$};

\path[->,font=\scriptsize,>=angle 90]
(A) edge node[above right]{$\lambda$} (D)
(A) edge node[left]{$\proy_M$} (C)
(A0) edge node[left]{$\proy_M|_{\Rr_M}$} (C0)
(C) edge node[below]{$\tilde{\lambda}$} (D);

\path[right hook->,font=\scriptsize,>=angle 90]
(A0) edge node[left]{} (A)
(C0) edge node[left]{} (C);
\end{tikzpicture}
\end{center}
The map $\tilde{\lambda}$ is continuous and bijective and it maps the compact set $\tilde{\overline{\Rr}}_M$ onto the Hausdorff space $\overline{\Tt}_M$, so it is a homeomorphism. Consequently, $\tilde{\overline{\Rr}}_M$ is homeomorphic to a disc and its boundary is $\proy_M(\{\rho=M\}\cup\{\theta=0\}\cup\{\theta=\tfrac{\text{\textpi}}{2}\})$, as required.
\qed
\end{proof}

\paragraph{}\hspace{-5mm}\label{prop4}
{\em Fix $B\ge A>0$ and consider the warped discs $\Dd_1$ and $\Dd_2$ introduced in \eqref{d1} and \eqref{d2}. Then there exists $M>0$ such that the boundary map $\partial\tilde{\phi}_M:\partial\tilde{\overline{\Rr}}_M\to\R^3$ meets transversally once both discs $\Dd_1$ and $\Dd_2$.}

\vspace{2mm}
\noindent\begin{proof} 
As $\Dd_1$ and $\Dd_2$ are bounded set, there exists $M_0>0$ such that $\Dd_1\cup\Dd_2\subset\{\|(x,y,z)\|<M_0\}$. Take $M:=4M_0$ and consider the set $\tilde{\overline{\Rr}}_M$ and the continuous map $\tilde{\phi}_M$ introduced in paragraph \ref{lem:dos0}.

\vspace{2mm}
We claim: \em the boundary map $\partial\tilde{\phi}_M:\partial\tilde{\overline{\Rr}}_M\rightarrow\R^3$ meets transversally once $\Dd_1$\em.

\begin{figure}[!ht]
\begin{center}
\includegraphics[height=6cm]{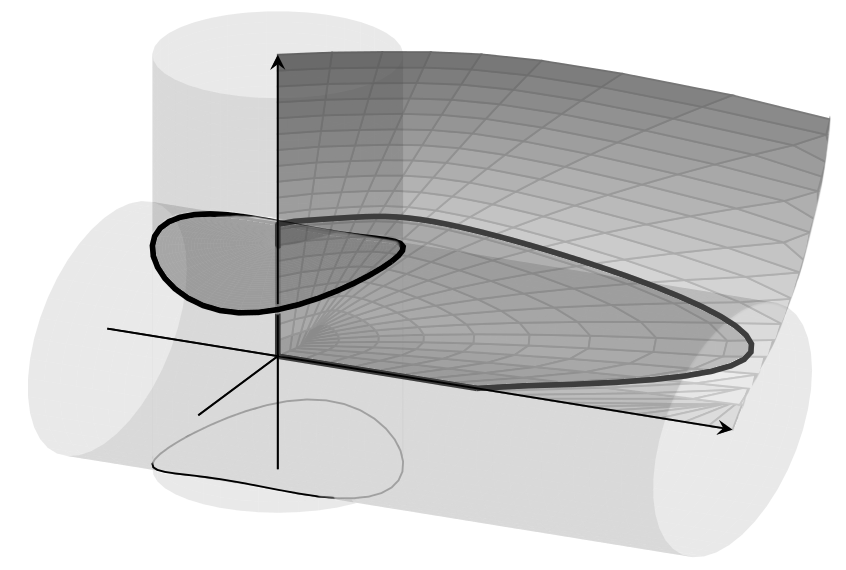}
\end{center}
\caption{The boundary map $\partial\tilde{\phi}_M:\partial\tilde{\overline{\Rr}}_M\rightarrow\R^3$ meets transversally once $\Dd_1$.}
\end{figure}

Consider the parameterization of $\partial\tilde{\overline{\Rr}}_M$ given by
$$
\beta_1(t):=
\begin{cases}
\proy_M(t,\frac{\text{\textpi}}{2}),&\text{if $0\le t\le M$},\\
\proy_M(M,M+\frac{\text{\textpi}}{2}-t),&\text{if $M<t\le M+\frac{\text{\textpi}}{2}$},\\
\proy_M(2M+\frac{\text{\textpi}}{2}-t,0),&\text{if $M+\frac{\text{\textpi}}{2}<t\le 2M+\frac{\text{\textpi}}{2}$}.
\end{cases}
$$ 
We have 
$$
\alpha_1(t):=\tilde{\phi}_M\circ\beta_1(t)=\begin{cases}
\phi(t,\frac{\text{\textpi}}{2}),&\text{if $0\le t\le M$},\\
\phi(M,M+\frac{\text{\textpi}}{2}-t),&\text{if $M<t\le M+\frac{\text{\textpi}}{2}$},\\
\phi(2M+\frac{\text{\textpi}}{2}-t,0),&\text{if $M+\frac{\text{\textpi}}{2}<t\le 2M+\frac{\text{\textpi}}{2}$}.
\end{cases}
$$ 

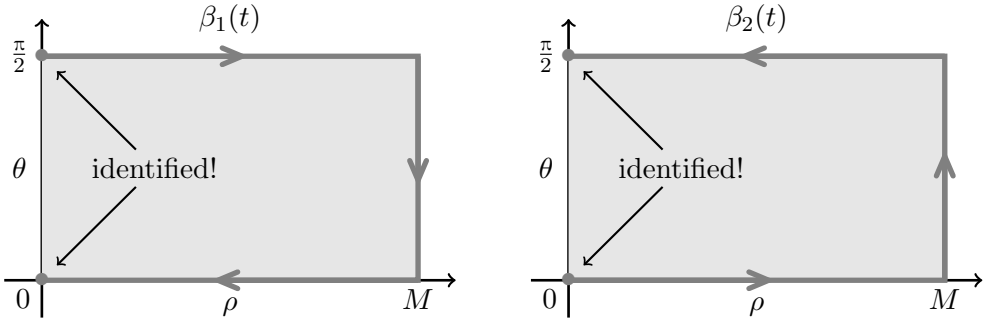
\begin{figure}[!ht]
\begin{center}
\begin{tikzpicture}[scale=0.99]
\draw[->, line width=1pt] (-0.5,0) -- (5.5,0);
\draw[->, line width=1pt] (0,-0.5) -- (0,3.5);
\draw[gray,line width=2pt,fill=gray!20] (0,3) to (5,3) to (5,0) to (0,0);
\draw[gray,line width=2pt,->,>=angle 45] (0,3) to (2.75,3);
\draw[gray,line width=2pt,->,>=angle 45] (5,3) to (5,1.25);
\draw[gray,line width=2pt,->,>=angle 45] (5,0) to (2.25,0);
\draw[->, line width=1pt] (6.5,0) -- (12.5,0);
\draw[->, line width=1pt] (7,-0.5) -- (7,3.5);
\draw[gray,line width=2pt,fill=gray!20] (7,3) to (12,3) to (12,0) to (7,0);
\draw[gray,line width=2pt,->,>=angle 45] (12,3) to (9.25,3);
\draw[gray,line width=2pt,->,>=angle 45] (12,0) to (12,1.75);
\draw[gray,line width=2pt,->,>=angle 45] (7,0) to (9.75,0);
\draw (-0.25,-0.25) node{$0$};
\draw (-0.3,1.5) node{$\theta$};
\draw (-0.3,3) node{$\tfrac{\text{\textpi}}{2}$};
\draw (2.5,-0.35) node{$\rho$};
\draw (5,-0.25) node{$M$};
\draw (0,0) node{{\color{gray}\large$\bullet$}};
\draw (0,3) node{{\color{gray}\large$\bullet$}};
\draw (1.5,1.5) node{identified!};
\draw[->,thick] (1.25,1.75) -- (0.2,2.8);
\draw[->,thick] (1.25,1.25) -- (0.2,0.2);
\draw (2.5,3.5) node{$\beta_1(t)$};
\draw (6.75,-0.25) node{$0$};
\draw (6.7,1.5) node{$\theta$};
\draw (6.7,3) node{$\tfrac{\text{\textpi}}{2}$};
\draw (9.5,-0.35) node{$\rho$};
\draw (12,-0.25) node{$M$};
\draw (7,0) node{{\color{gray}\large$\bullet$}};
\draw (7,3) node{{\color{gray}\large$\bullet$}};
\draw (8.5,1.5) node{identified!};
\draw[->,thick] (8.25,1.75) -- (7.2,2.8);
\draw[->,thick] (8.25,1.25) -- (7.2,0.2);
\draw (9.5,3.5) node{$\beta_2(t)$};
\end{tikzpicture}
\end{center}
\caption{Behavior of the paths $\beta_1$ and $\beta_2$.}
\end{figure}

Choose $0<\varepsilon<\min\{B,M_0-B\}$ and consider the homeomorphism
$$
\zeta_1:\R^3\to\R^3,\ (x,y,z)\mapsto(x,y,z-\xi_1(x,y)),
$$
where $\xi_1$ is the (continuous) semialgebraic function introduced in \eqref{xi1}. Denote $\Dd_1(\varepsilon):=\zeta_1^{-1}({\mathbb D}_A(\varepsilon))$. It is enough to check:
$$
\alpha_1^{-1}(\Dd_1(\varepsilon))=(B-\varepsilon,B+\varepsilon). 
$$

Pick $p_0:=\alpha_1(t_0)\in\text{Im}(\alpha_1)$. We distinguish three cases:

\begin{itemize}
\item[(i)] If $0\le t_0\le M$, then $\zeta_1(p_0)=(\zeta_1\circ\phi)(t_0,0)=(0,0,t_0-B)$. Consequently, $\zeta_1(p_0)\in{\mathbb D}_A(\varepsilon)$ if and only if $-B<-\varepsilon<t_0-B<\varepsilon<M-B$. 

\item[(ii)] If $M<t_0\le M+\frac{\text{\textpi}}{2}$, we have by \eqref{eq:1}
$$
\dist(p_0,{\bf0})\geq\tfrac{M}{2}=2M_0>\sqrt{2}M_0>\dist(q,{\bf0})
$$ 
for each $q\in\Dd_1(\varepsilon)$. Therefore $p_0\notin\Dd_1(\varepsilon)$.

\item[(iii)] If $M+\frac{\text{\textpi}}{2}<t_0\le 2M+\frac{\text{\textpi}}{2}$, then 
$$
p_0=\alpha_1(t_0)=\phi(2M+\tfrac{\text{\textpi}}{2}-t_0,0)=(2M+\tfrac{\text{\textpi}}{2}-t_0,0,0),
$$ 
so $\zeta_1(p_0)=(2M+\tfrac{\text{\textpi}}{2}-t_0,0,-B)$. As $\varepsilon<B$, it holds $\zeta_1(p_0)\not\in{\mathbb D}_A(\varepsilon)$, so $p_0\not\in\Dd_1(\varepsilon)$.
\end{itemize}
We conclude $\alpha_1^{-1}(\Dd_1(\varepsilon))=(B-\varepsilon,B+\varepsilon)$, so $\alpha_1$ meets transversally once $\Dd_1$. 

\vspace{2mm}
Analogously one shows: \em the boundary map $\partial\tilde{\phi}_M:\partial\tilde{\overline{\Rr}}_M\rightarrow\R^3$ meets transversally once $\Dd_2$\em.

\begin{figure}[!ht]
\begin{center}
\includegraphics[height=6cm]{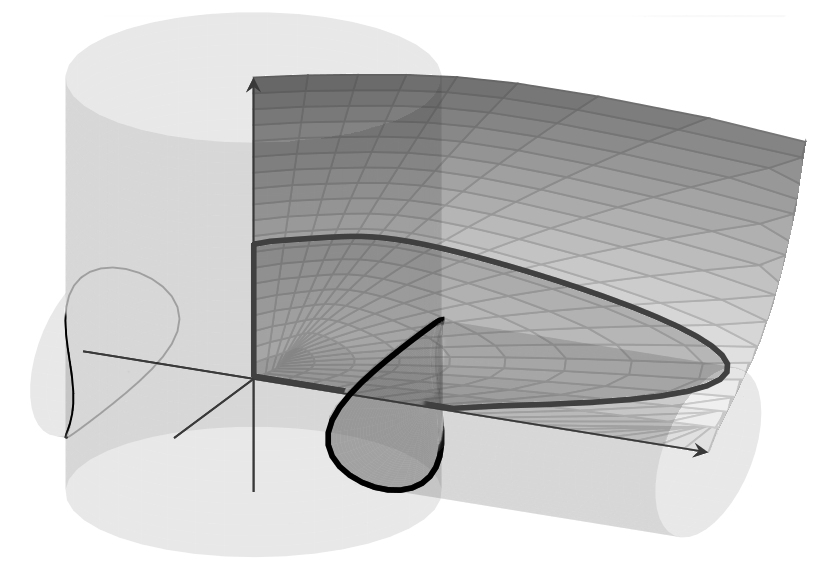}
\end{center}
\caption{The boundary map $\partial\tilde{\phi}_M:\partial\tilde{\overline{\Rr}}_M\rightarrow\R^3$ meets transversally once $\Dd_2$.}
\end{figure}

Consider in this case the parameterization of $\partial\tilde{\overline{\Rr}}_M$ given by
$$
\beta_2(t):=
\begin{cases}
\proy_M(t,0),&\text{if $0\le t\le M$},\\
\proy_M(M,t-M),&\text{if $M<t\le M+\frac{\text{\textpi}}{2}$},\\
\proy_M(2M+\frac{\text{\textpi}}{2}-t,\frac{\text{\textpi}}{2}),&\text{if $M+\frac{\text{\textpi}}{2}<t\le 2M+\frac{\text{\textpi}}{2}$}.
\end{cases}
$$ 
We have
$$
\alpha_2(t):=\tilde{\phi}_M\circ\beta_2(t)=
\begin{cases}
\phi(t,0),&\text{if $0\le t\le M$},\\
\phi(M,t-M),&\text{if $M<t\le M+\frac{\text{\textpi}}{2}$},\\
\phi(2M+\frac{\text{\textpi}}{2}-t,\frac{\text{\textpi}}{2}),&\text{if $M+\frac{\text{\textpi}}{2}<t\le 2M+\frac{\text{\textpi}}{2}$}.
\end{cases}
$$ 

Proceed as above keeping the same values for $A$ and $\varepsilon$ and using in this case the homeomorphism
$$
\zeta_2:\R^3\to\R^3,\ (x,y,z)\mapsto(z,y,x-\xi_2(z,y)),
$$
where $\xi_2$ is the (continuous) semialgebraic function introduced in \eqref{xi2}, to prove that $\alpha_2$ meets transversally once the warped disk $\Dd_2$.\qed
\end{proof}

\paragraph{}\hspace{-3mm}By \ref{lem:dos} $\overline{\Rr}_M$ is homeomorphic to a closed disc. By Proposition~\ref{prop:top} applied to the continuous map $\phi_M:\overline{\Rr}_M\to\R^3$ and \ref{prop4}, we deduce that the boundaries of both warped discs $\Dd_1$ and $\Dd_2$ meet $\phi_M({\Rr}_M)\subset\Ss$. Thus, \ref{311} holds, as required.\qed

\end{document}